\documentclass[11pt]{amsart}
\usepackage{amssymb, amsmath}
\usepackage{graphics}
\usepackage[dvips]{graphicx}
\usepackage{eepic}
\usepackage{epic}
\newtheorem{thm}{Theorem}[section]

\newtheorem{lmm}[thm]{Lemma}

\newtheorem{prp}[thm]{Proposition}

\theoremstyle{definition}
\newtheorem{dfn}[thm]{Definition}
\theoremstyle{remark}
\newtheorem*{rem}{Remark}

\begin{document}
\title{On commensurability of quadratic differentials}

\author{Hidetoshi Masai}
\address{Mathematical Science Group of Advanced Institute for Materials Research, Tohoku University, 2-1-1, Katahira, Aoba-ku, Sendai, 980-8577, Japan}
\email{masai@tohoku.ac.jp}
\thanks{To appear in RIMS K\^oky\^uroku Bessatsu. This is a preprint version. This paper was written and accepted while the author was in University of Tokyo.}
%

\maketitle

\begin{abstract}      
We consider commensurability of quadratic differentials on surfaces.
Each commensurability class has a natural order by the covering relation.
We show that each commensurability class contains a unique (orbifold) element.
We also discuss the relationship between commensurability of quadratic differentials and fibered commensurability, a notion introduced by Calegari-Sun-Wang.
\end{abstract}

\section{preface}
Let $S$ be an orientable surface of finite type and $X$ a Riemann surface of finite analytic type which is homeomorphic to $S$.
Throughout the paper, surfaces considered are orientable and of finite type with $3g(S)-3+p(S)>0$ 
where $g(S)$ is the genus and $p(S)$ is the number of punctures.
A (holomorphic) {\em quadratic differential} $q = \{(U_{j},z_{j}),q_{j}\}_{j}$ 
is a family of holomorphic functions $\{q_{j}\}_{j}$, one defined on each local chart $(U_{j},z_{j})$ of $X$ so that if
$U_{i}\cap U_{j}\not=\emptyset$, then $q_{j}(z_{j}) = q_{i}(z_{i})(dz_{i}/dz_{j})^{2}$.
Quadratic differentials are arrowed to have poles of degree one at the punctures.
A quadratic differential on a topological surface $S$ is 
a quadratic differential with respect to some complex structure on $S$.
For a given homeomorphism $f:S\rightarrow S$ and a quadratic differential $q=\{(U_{j},z_{j}),q_{j}\}_{j}$, 
we define push-forward by 
$$f_{*}(q) := \{(f(U_{j}), z_{j}\circ f^{-1}),q_{j}\}_{j}.$$
Given a quadratic differential $q$, by integrating the square root of $q$, we get natural coordinates around nonzero points of $q$.
This gives a singular Euclidean structure on $S$.
Furthermore $q$ has natural horizontal (resp. vertical) directions which defined to be $v \in T_{z}X$ so that $q(z)v^{2}$ is a positive (resp. negative) real number.
In this paper, coverings considered are assumed to be unramified.
Let $p:\widetilde S\rightarrow S$ be a finite covering.
Since each quadratic differential is defined on charts, we can naturally lift quadratic differentials via any finite covering.
For a quadratic differential $q$ on $S$, its lift with respect to $p:\widetilde S\rightarrow S$ is denoted by $p^{*}q$.
Conversely, a quadratic differential $\widetilde{q}$ on a Riemann surface $\widetilde S$
is said to be {\em symmetric} with respect to $p:\widetilde S\rightarrow S$ if 
$\widetilde{q}$ is a lift of a quadratic differential $q$ on $S$.
If $\widetilde{q}$ is symmetric, 
we use the notation of push-forward for coverings, i.e. $q = p_{*}(\widetilde{q})$.
Note that if $\widetilde{q}$ is symmetric, then every quadratic differential on the Teichm\"uller geodesic determined by $\widetilde{q}$ is symmetric (c.f. \cite{GM}).
\begin{dfn}
Let $q_{1}$ and $q_{2}$ be quadratic differentials on surfaces $S_{1}$ and $S_{2}$ respectively.
The quadratic differential $q_{1}$ is said to {\em cover} $q_{2}$ if there exists a finite covering $p:S_{1} \rightarrow S_{2}$
so that $p^{*}q_{2} = q_{1}$.
\end{dfn}
Once we have a covering relation, we may consider commensurability.	
\begin{dfn}
Two quadratic differentials $q_{1}$ and $q_{2}$ on surfaces $S_{1}$ and $S_{2}$ respectively are {\em commensurable} (denoted $q_1\sim q_2$)
if there is a third quadratic differential $\widetilde{q}$ that covers both $q_{1}$ and $q_{2}$.
\end{dfn}
We also need to consider orbifolds.
On a 2-dimensional orientable orbifold, quadratic differentials are defined on the surface that we get by puncturing every orbifold point. 
When we take coverings, we fill every puncture corresponding to an orbifold point once it is locally covered by a surface.
It can be seen that this commensurability is an equivalence relation (Proposition \ref{prop.equiv}).
Two quadratic differentials $q_{1},q_{2}$ are said to be {\em conjugate} if there exists a surface homeomorphism (or a orbifold automorphism) $f$
such that $f_{*}q_{1} = q_{2}$.
By considering conjugacy classes, we endow each commensurability class with an order by the covering relation.
The main theorem of this paper is the following.
\begin{thm}\label{thm.main}
Every commensurability class of quadratic differentials contains a unique minimal (orbifold) element.
\end{thm}

\begin{rem}
Theorem \ref{thm.main} can also be verified by the identity theorem if there is a zero of the quadratic differentials in the interior of the surface.
The argument in this paper works even for the case where all the zeros are punctured.
\end{rem}

We also consider fibered commensurability and its relationship with the commensurability of quadratic differentials.
In \cite{CSW}, Calegari-Sun-Wang introduced following commensurability on surface automorphisms.
Let $\mathrm{Mod}(S)$ denote the mapping class group of $S$.

\begin{dfn}[\cite{CSW}]
Let $S_1$ and $S_2$ be orientable surfaces of finite type.
A mapping class $\phi_1\in\mathrm{Mod}(S_1)$ {\em covers} $\phi_2\in\mathrm{Mod}(S_2)$ if
there exists a finite covering $p:S_1\rightarrow S_2$ and 
$k\in\mathbb{Z}\setminus\{0\}$ such that a lift $\varphi$ of $\phi_2^{k}$ with respect to $p$ satisfies 
$\varphi = \phi_1$.
Two mapping classes are said to be {\em commensurable}
if there exists a mapping class that covers both.
\end{dfn}

We call an orbifold automorphism psuedo-Anosov if it can be covered by a psuedo-Anosov mapping class on a surface.
In \cite{CSW} and \cite{Masa}, the following theorem was shown.
\begin{thm}\label{thm.fibered}
Every fibered commensurability class of pseudo-Anosov mapping classes contains a unique minimal (orbifold) element.
\end{thm}
We give a new proof of the above theorem by using Theorem \ref{thm.main}.

\section{Unique minimal element}
We first observe that commensurability of quadratic differentials is an equivalence relation.
The reflectivity and symmetry are trivial.
\begin{prp}\label{prop.equiv}
Let $q_i$ be a quadratic differential on a surface $S_i$ for $i=1,2,3$.
Suppose $q_1\sim q_2$ and $q_2\sim q_3$.
Then we have $q_1\sim q_3$.
\end{prp}
\begin{proof}
Let $q_{j,j+1}$ be a quadratic differential on a surface $S_{j,j+1}$ which covers $q_j$ and $q_{j+1}$ for $j=1,2$.
Then in the orbifold fundamental group $\pi_1(S_2)$, the images of $\pi_1(S_{1,2})$ and $\pi_1(S_{2,3})$ by the covering maps are of finite index.
Hence the intersection $H:= \pi_1(S_{1,2})\cap \pi_1(S_{2,3})$ in $\pi_1(S_2)$ is also a finite index subgroup.
Then the lift $\widetilde{q_2}$ of $q_2$ to the covering corresponding to $H$ is also the lift of $q_{1,2}$ and $q_{2,3}$.
Therefore $\widetilde{q_2}$ covers all $q_1$, $q_2$, and $q_3$, in particular $q_1\sim q_3$.
\end{proof}

The main idea of the following proof of Theorem \ref{thm.main} is from \cite[Lemma 4.11]{Masa2}.
\begin{proof}[Proof of Theorem \ref{thm.main}]
We show that if $q_1\sim q_2$, then both $q_1$ and $q_2$ cover the same quadratic differential $q'$.
Recall that each quadratic differential $q$ determines a singular Euclidean structure with horizontal and vertical foliation.
Let $\mathrm{Sing}(q)$ denote the set of singular points of the singular Euclidean structure.
This $\mathrm{Sing}(q)$ is finite and contains all punctures.
Let $q_{1,2}$ be a quadratic differential on a surface $S_{1,2}$ which covers both $q_1$ and $q_2$ and 
let $p_i:S_{1,2}\rightarrow S_i$  denote the associated covering maps for $i=1,2$.
Pick any $s\in\mathrm{Sing}(q_{1,2})$, then we define $\Sigma_1(s):=p_{1}^{-1}(p_{1}(s))$.
Inductively define $\Sigma_{i+1}(s):= p_{[i+1]}^{-1}p_{[i+1]}(\Sigma_{i})$ where $[k] = 1$ if $k$ is odd and $[k]=2$ if $k$ is even.
Since $\Sigma_i(s)\subset \Sigma_{i+1}(s)\subset \mathrm{Sing}(q_{1,2})$, 
we eventually have $\Sigma_i(s) = \Sigma_{i+1}(s) (=:\Sigma(s))$ for large enough $i$.
Thus we have an equivalence relation on $\mathrm{Sing}(q_{1,2})$.
Next, we pick any $x\in S_{1,2}\setminus \mathrm{Sing}(q_{1,2})$.
There is a point $s'\in\mathrm{Sing}(q_{1,2})$ such that we can connect $x$ and $s'$ by a single Euclidean geodesic $\gamma$.
The geodesic $\gamma$ has well defined angle $\theta_\gamma\mod \pi$.
Let $l_q(\gamma)$ denote the Euclidean length of $\gamma$.
Since there are only finitely many points from $\Sigma(s')$ with angle $\theta_\gamma$ and Euclidean distance $l_q(\gamma)$,
we get $\Sigma(x)\subset S\setminus\mathrm{Sing}(q)$ in the same way as above.
Thus we get an equivalence relation $x\sim y:\iff y\in\Sigma(x)$ on $S_{1,2}$. 
Since this relation is defined by composing local homeomorphisms $p_{1}$ and $p_{2}$, the quotient map $p':S_{1,2} \rightarrow S/{\sim}$ is a covering.
Note that $S/{\sim}$ might be an orbifold.
For each point $x\in S/{\sim}$, we may find a small open neighborhood $U_{x}$ so that on all pre-images in $S_{1,2}$,
the quadratic differentials can be identified via $p_{1}$ and $p_{2}$.
Hence $p'$ determines a quadratic differential $q'$ on $S/{\sim}$.
By construction, $p'$ factors through $p_{i}:S_{1,2} \rightarrow S_{i}$ and hence both $q_{1}$ and $q_{2}$ cover $q'$.
If there is another quadratic differential $q_{3}$ in the commensurability class that does not cover $q'$, 
then we may apply the same argument as above to find a quadratic differential which is covered by $q'$ and $q_{3}$.
Since each time we get a new quadratic differential, the Euler characteristic of the underlying orbifold decreases, 
this process would terminate.
Thus we get a unique minimal element.
\end{proof}

\section{Fibered commensurability}
In this section, we give a proof of Theorem \ref{thm.fibered}.
\begin{lmm}\label{lem.lift}
Let $p:\widetilde{S}\rightarrow S$ be a finite covering of orientable 2-orbifolds.
Let $f:\widetilde{S}\rightarrow \widetilde{S}$ be an orbifold automorphism.
If there is a quadratic differential $q$ such that $f^{n}_{*}(q)$ is symmetric with respect to 
$p:\widetilde{S}\rightarrow S$ for each $n\in\mathbb{Z}$.
Then there is a finite covering $p':\widetilde{S}\rightarrow S'$ which factors through $p$ such that
$f$ is a lift of some homeomorphism $f':S'\rightarrow S'$.
\end{lmm}
\begin{proof}
Note that all $(p\circ f^{n})_{*}(q)$ are commensurable to each other.
We define $x\sim y$ if there exists $n$ such that $p\circ f^{n}(x)=p\circ f^{n}(y)$.
We take transitive closure of this $\sim$ to define an equivalence relation.
Let $\Sigma(x)$ denote the equivalence class of $x$.
Similarly to the proof of Theorem \ref{thm.main}, we see that $\Sigma(x)$ is finite and defines a covering.
Furthermore, by construction, we see that $f(\Sigma(x)) = \Sigma(f(x))$.
Thus we have a homeomorphism $f':S'\rightarrow S'$ so that $f$ is a lift of $f'$.
\end{proof}
\begin{proof}[Proof of Theorem \ref{thm.fibered}]
For $i = 1,2$, let $f_{i}:S_{i}\rightarrow S_{i}$ be commensurable pseudo-Anosov homeomorphisms.
We will find $f':S'\rightarrow S'$ which is covered both by $f_{1}$ and $f_{2}$.
Note that if two pseudo-Anosov maps are commensurable, 
then they  give commensurable quadratic differentials whose horizontal and vertical foliations are
stable and unstable foliations \cite{GM}.
Let $q_{i}$ be a quadratic differential on $S_{i}$ associated to $f_{i}$ for $i = 1,2$ so that $q_{1}$ and $q_{2}$ are commensurable.
Then by Theorem \ref{thm.main}, there is a unique minimal quadratic differential $q'$ on some orbifold $S'$ in the commensurability class.
Note that for all $i=1,2$ and $n\in\mathbb{Z}$,
$(f_{i}^{n})_{*}q_{i}$ is also symmetric with respect to the covering from $S_{i}$ to $S'$.
This is because $f_{i}$ preserves projective classes of horizontal and vertical foliations of $q_{i}$.
Hence if $f_{i}$ is not any lift of a homeomorphism on $S'$, we may find further covering from $q'$ by Lemma \ref{lem.lift}.
This contradicts the minimality of $q'$, so $f_{i}$ is a lift of some homeomorphism $f_{i}'$ on $S'$ for both $i=1,2$.
The stable and unstable foliations of $f'_{1}$ and $f'_{2}$ agree.
Note that there is a lower bound for the stretch factor of pseudo-Anosovs.
Therefore there are integers $a_{i}$ such that $a_{1}\log\lambda(f_{1}) = a_{2}\log\lambda(f_{2})$ where $\lambda(f)$ is the stretch factor of $f$.
Let $n_{i}~(i=1,2)$ be integers such that $n_{1}a_{1} + n_{2}a_{2}$ is the greatest common divisor of $a_{1}$ and $a_{2}$.
Then by considering the stretch factor, we see that $f':=(f_{1}')^{n_{1}}\circ (f_{2}')^{n_{2}}$ satisfies
that there are integers $m_{i}$ such that $(f')^{m_{i}}\circ f_{i}'$ preserves $q'$.
By appealing to the singular Euclidean structure with respect to $q'$, it can readily be seen that $(f')^{m_{i}}\circ f_{i}'$ is of finite order.
But if $(f')^{m_{i}}\circ f_{i}'$ is non-trivial, it contradicts the minimality of $q'$.
Hence $f'$ is a root of both $f_{1}'$ and $f_{2}'$.
Therefore $f'$ is covered by both $f_{1}$ and $f_{2}$.
Similarly to the proof of Theorem \ref{thm.main}, after finite steps, we end up with a unique minimal element.
\end{proof}
\section*{Acknowledgements}
The author would like to thank Michihiko Fujii for the invitation to the conference.
He also thanks Hideki Miyachi for helpful conversations.
Finally he thanks anonymous referee for useful comments.


\end{document}